\documentclass[11pt, a4paper]{article}

\usepackage{amsmath, amssymb, amsthm, xspace, a4wide, amscd}
\usepackage[english]{babel}
\usepackage[T1]{fontenc}
\usepackage[latin1]{inputenc}

\newcommand{\ff}{\mathcal{F}}
\newcommand{\bb}{\mathcal{B}}
\newcommand{\rr}{\mathbb{R}}
\newcommand{\zz}{\mathbb{Z}}
\newcommand{\cc}{\mathbb{C}}
\newcommand{\pp}{{\sf{P}}}
\newcommand{\xx}{{\sf X}}
\newcommand{\ee}{{\sf E}}
\newcommand{\pr}{\partial}
\newcommand{\llf}{{\sf L}}
\newcommand{\ii}{\mathbf{1}}

\newcommand{\aaa}{\mathcal{A}}
\newcommand{\hh}{\mathcal{H}}

\newtheorem{thm}{Theorem}
\newtheorem{lemma}{Lemma}
\newtheorem{assx}{Assumption A\hspace*{-4pt}}

\theoremstyle{definition}
\newtheorem{defin}{Definition}

\theoremstyle{remark}

\begin{document}

\begin{center}\Large
The Burgers equation driven by a stochastic measure
\end{center}

\begin{center}
Vadym Radchenko\footnote{Department of Mathematical Analysis, Taras Shevchenko National University of Kyiv,
01601 Kyiv, {Ukraine},\\ vadymradchenko@knu.ua}
 \footnote{This work was supported by Alexander von Humboldt Foundation, grant 1074615. The author is grateful to Prof. M.~Z\"{a}hle for fruitful discussions during the preparation of this paper and thanks the Friedrich-Schiller-University of Jena for its hospitality.}
\end{center}

\begin{abstract}
We study the class of one-dimensional equations driven by a stochastic measure $\mu$. For $\mu$ we assume only $\sigma$-additivity in probability. This class of equations include the Burgers equation and the heat equation. The existence and uniqueness of the solution are proved, and the averaging principle for the equation is studied.
\end{abstract}

\section{Introduction}\label{scintr}

In this paper, we consider the stochastic equation, which can formally be written as
\begin{equation}\label{eq:burgg}
\frac{\pr u}{\pr t}=\frac{\pr^2 u}{\pr x^2}+f(t,x,u(t,x))+\frac{\pr g}{\pr x}(t,
x,u(t,x)) + \sigma(t, x)\frac{\pr  \mu}{\pr x},\quad u(0,x)=u_0(x)
\end{equation}
where $(t, x)\in [0,\ T]\times\rr$, and $\mu$ is a stochastic measure defined on $\bb(\rr)$ (Borel
$\sigma$-algebra in $\rr$). For $\mu$ we assume only $\sigma-$additivity in probability, assumptions for $f,\ g,\ \sigma$
and $u_0$ are given in Section~\ref{sc:prob}. We consider the solution to the formal equation~\eqref{eq:burgg} in the mild
form (see~\eqref{eq:bemfg} below).

If $f=0$ and $g(x,t,v)=v^2/2$ then \eqref{eq:burgg} is the Burgers equation, for $g={\rm const}$ we get the heat equation.

In the paper, we prove the existence and uniqueness of the solution and obtain its $\llf^2$-continuity. In addition, we obtain that the averaging principle holds for equation~\eqref{eq:burgg}.

Equation~\eqref{eq:burgg} driven by the Wiener process was studied in \cite{gyonua} for $x\in\rr$, in \cite{gyongy98} and \cite{gyonrov} for $x\in[0,1]$. Equation driven by L\'{e}vy process was considered in \cite{jacob10}. The main reason to study equation~\eqref{eq:burgg} is that it is a generalisation of the Burgers equation which is very important in fluid mechanics. The stochastic Burgers equation was studied, for example, in \cite{dong07}, \cite{lewnua18}, \cite{mazz20}, \cite[Chapter 18]{peszab}, \cite{yuan22}, \cite{zhou22}. All these equations are driven by Gaussian processes or L\'{e}vy processes. In our paper, we consider a more general integrator. Our proofs are based on methods and results of~\cite{gyonua}.

Stochastic equations driven by stochastic measures were studied, for example, in \cite{bodumj} (general parabolic equation), \cite{bodrad20} (wave equation), \cite{rads09} (heat equation). The averaging principle for such equations was considered in \cite{bodnarchuk_2020},~\cite{radavsm19}, \cite{manikin_2022av}. The detailed theory of stochastic measures is presented in~\cite{radbook}. Note that in all these publications, the functions in the equations are assumed to be bounded.

The rest of the paper is organized as follows. In Section~\ref{sc:prel} we have compiled some basic facts about stochastic
measures. The precise formulation of the problem and our assumptions are given in Section~\ref{sc:prob}, some regularity properties of the stochastic integral are studied in Section~\ref{sc:esti}, and one auxiliary equation is considered in Section~\ref{sc:burgaux}. The existence and uniqueness of the solution to equation~\eqref{eq:burgg} are proved in Section~\ref{sc:burgsol}. The averaging principle for our equation is obtained in Section~\ref{sc:aver}.

\section{Preliminaries}\label{sc:prel}

In this section, we give basic information concerning stochastic measures in a general setting. In equation~\eqref{eq:burgg}, $\mu$ is defined on Borel subsets of $\rr$.

Let $\llf_0=\llf_0(\Omega, {\ff}, {\pp} )$ be the set of all real-valued
random variables defined on the complete probability space $(\Omega, {\ff}, {\pp} )$. Convergence in $\llf_0$ means the convergence in probability. Let ${\xx}$ be an arbitrary set and ${\bb}$ a $\sigma$-algebra of subsets of ${\xx}$.

\begin{defin}\label{def:stme}
A $\sigma$-additive mapping $\mu:\ {\bb}\to \llf_0$ is called {\em stochastic measure} (SM).
\end{defin}

We do not assume the moment existence or martingale properties for SM. In other words, $\mu$ is $\llf_0$--valued measure.

Important examples of SMs are orthogonal stochastic measures, $\alpha$-stable random measures defined on a $\sigma$-algebra for $\alpha\in (0,1)\cup(1,2]$ (see \cite[Chapter 3]{samtaq}). Theorem~8.3.1 of~\cite{kwawoy} states the conditions under which the increments
of a real-valued process with independent increments generate an SM.

Many examples of the SMs on the Borel subsets of $[0, T]$ may be given by the Wiener-type integral
\begin{equation}\label{eq:muax}
\mu(A)=\int_{[0,T]} {\ii}_A(t)\,{d}X_t.
\end{equation}

We note the following cases of processes $X_t$ in~\eqref{eq:muax} that generate SM.

\begin{enumerate}

\item\label{itmart} $X_t$~-- any square integrable martingale.

\item\label{itfrbr} $X_t=W_t^H$~-- the fractional Brownian motion with Hurst index $H>1/2$, see Theorem~1.1~\cite{memiva}.

\item\label{itsfrb} $X_t=S_t^k$~-- the sub-fractional Brownian motion for $k=H-1/2,\ 1/2<H<1$, see Theorem 3.2~(ii) and Remark 3.3~c) in~\cite{tudor09}.

\item\label{itrose} $X_t=Z_H^k(t)$~-- the Hermite process, $1/2<H<1$, $k\ge 1$, see~\cite{tudor07}, \cite[Section 3.1.3]{tudor13}. $Z_H^2(t)$ is known as the Rosenblatt process, see~\cite[Section~3]{tudor08}.

\end{enumerate}

The following analogue of the Nikodym theorem is satisfied for SMs.

\begin{thm}(\cite[Theorem 8.6]{dretop})\label{thniko}
Let $\mu_n$ are SMs on ${\bb}$, $n\ge 1$, and
\[
\forall\, {A}\in{\bb}\ \exists\ \mu({A})={\rm p}\lim_{n\to\infty} \mu_n ({A}).
\]
Then $\mu$ is a SM on ${\bb}$.
\end{thm}

Applying Theorem~\ref{thniko}, we can construct the following example of the SM on all Borel subsets of $\rr$:
\begin{equation}\label{eq:muaxw}
\mu(A)=\int_{\rr} {\ii}_A(t)\xi(t)\,{d}W_t:=\lim_{T\to \infty} \int_{[0,T]} {\ii}_A(t)\xi(t)\,{d}W_t,
\end{equation}
where $W_t$ is the Wiener process, $\xi(t)$ is an adapted process such that $\int_{\rr}\ee \xi(t)^2\,dt<\infty$, the limit is taken in~$\llf^2(\Omega)$.

For deterministic measurable functions $g:\xx\to\rr$, an integral of the form $\int_{\xx}g\,d\mu$ is studied
in~\cite[Chapter 1]{radbook} (see also~\cite[Chapter 7]{kwawoy}). In particular, every bounded measurable $g$ is integrable w.~r.~t. any SM~$\mu$.

To estimate the stochastic integral, we will use the following statement.

\begin{lemma} (Lemma 3.1~\cite{rads09}, Lemma 2.3~\cite{radbook})\label{lm:fkmu}
Let $f_l:\ {\xx}\to \rr,\ l\ge 1$, be measurable functions such that
$ \bar{f}(x)=\sum_{l=1}^{\infty} |{f_l}(x)|$ is integrable w.r.t.~$\mu$. Then
\begin{equation*}
\sum_{l=1}^{\infty}\Bigl(\int_{\xx} f_l\,d\mu \Bigr)^2<\infty\quad \mbox{\textrm ~a.~s.}
\end{equation*}
\end{lemma}

We consider the Besov spaces $B^\alpha_{22}([c, d])$. Recall that the norm in this classical space for
$0<\alpha< 1$ may be introduced by
\begin{equation}
\label{eq:bssn}
\|g\|_{B^\alpha_{22}([c, d])}=\|g\|_{\llf_{2}([c, d])}+\Bigl(\int_0^{d-c}\ {(w_2(g, r))^2}{r^{-2\alpha-1}}
\,dr\Bigr)^{1/2},
\end{equation}
where
\[
w_2(g, r)=\sup_{0\le h\le r}\Bigl(\int_{c}^{d-h} |g(y+h)-g(y)|^2\,dy\Bigr)^{1/2}.
\]

By $C$ and $C(\omega)$ we will denote positive constants and positive random constants respectively whose values may change.

\section{The problem}
\label{sc:prob}

We consider~\eqref{eq:burgg} in the mild form
\begin{equation}\label{eq:bemfg}
\begin{split}
u(t,x)=\int_{\rr} p(t,x-y)u_0(y)\,dy
+\int_0^t \int_{\rr}p(t-s,x-y) f(s,y,u(s,y))\,dy\,ds\\
-\int_0^t \int_{\rr}\frac{\pr p}{\pr y}(t-s,x-y) g(s,y,u(s,y))\,dy\,ds
+\int_{\rr} \int_0^t p(t-s,x-y) \sigma(s,y)\,ds\,d\mu(y),
\end{split}
\end{equation}
where equality holds a.~s. for each $t\in [0,T]$ for almost all $x\in\rr$, $u\in\cc([0,T],\llf^2(\rr))$. Here
\[
{p}(t, x) = \frac{1}{2\sqrt{\pi t}}\, e^{-\frac{x^2}{4t}}
\]
is the heat kernel, and $\mu$ is an SM on ~$\bb(\rr)$.

We will refer to the following assumptions on elements of~\eqref{eq:bemfg}.

\begin{assx}\label{assxu}$u_0(y)=u_0(y, \omega):{\rr}\times\Omega\to\rr$ is measurable and $u_0(y)=u_0(y, \omega)\in \llf^2(\rr)$ for each fixed $\omega$.
\end{assx}

\begin{assx}\label{assxfall}  $f(s, y, r):[0,T]\times{\rr}^2\to\rr$ is a Borel function satisfying the linear growth and Lipschitz condition:
\begin{eqnarray*}
|f(s, y, r)|\le a_1(y)+K|r|,\\
|f(s, y, r_1)-f(s, y, r_2)|\le (a_2(y)+L|r_1|+L|r_2|)|r_1-r_2|
\end{eqnarray*}
for all $s\in[0,T]$, $y,r_i\in\rr$ and for some constants $K,L$ and non-negative functions $a_i\in\llf^2(\rr)$.
\end{assx}

\begin{assx}\label{assxgall}  The function $g$ is of the form
\[
g(s,y,r)=g_1(s,y,r)+g_2(s,r),
\]
where
\[
g_1(s,y,r):[0,T]\times{\rr}^2\to\rr,\quad g_2(s,r):[0,T]\times{\rr}\to\rr
\]
are Borel functions satisfying the following linear and quadratic growth conditions:
\begin{eqnarray*}
|g_1(s, y, r)|\le b_1(y)+b_2(y)|r|,\\
|g_2(s, r)|\le K|r|^2,
\end{eqnarray*}
for all $s\in[0,T]$, $y,r\in\rr$, where $K$ is a constant, and
\[
b_1\in\llf^1(\rr)\cap\llf^2(\rr),\quad b_2\in\llf^2(\rr)\cap\llf^\infty(\rr)
\]
are nonnegative functions. Moreover, $g$ satisfies the following Lipschitz condition:
\begin{equation*}
|g(s, y, r_1)-g(s, y, r_2)|\le (b_3(y)+L|r_1|+L|r_2|)|r_1-r_2|
\end{equation*}
for all $s\in[0,T]$, $y,r_i\in\rr$ and for some constant $L$ and non-negative function $b_3\in\llf^2(\rr)$.
\end{assx}

\begin{assx}\label{assxs}
$\sigma(s,y):[0, T]\times{\rr}\to\rr$ is measurable, and
\[
|\sigma(s,y)|\le C_{\sigma},\quad |\sigma(s,y_1)-\sigma(s,y_2)|\le L_{\sigma}|y_1-y_2|^{\beta(\sigma)}
\]
for $1/2<\beta(\sigma)<1$ and some constants $C_{\sigma},\ L_{\sigma}$.
\end{assx}

\begin{assx}\label{assxm} $y^{\tau}$ is integrable w.~r.~t. $\mu$ on~$\rr$ for some $\tau>1/2$.
\end{assx}

For example, Assumption A\ref{assxm} holds for SM defined in \eqref{eq:muaxw} and given $\tau$ if
\[
\int_{\rr}t^{2\tau}\ee \xi(t)^2\,dt<\infty.
\]

If A\ref{assxs} and A\ref{assxm} hold, then, by Theorem of~\cite{bodumj} (or Theorem~3.1~\cite{radbook}), the random function
\[
\vartheta(t,x)=\int_{\rr} \int_0^t p(t-s,x-y) \sigma(s,y)\,ds\,d\mu(y)
\]
has a version $\widetilde{\vartheta}(t,x)$ such that for any fixed
\[
\delta\in (0,T),\quad M>0,\quad 0<\gamma_1<1/2,\quad 0<\gamma_2<1/4
\]
there exists $L_{\widetilde{\vartheta}}(\omega)$ such that
\begin{equation}\label{eq:vervh}
\begin{split}
|\widetilde{\vartheta}(t_1,x_1)-\widetilde{\vartheta}(t_2,x_2)|
\le L_{\widetilde{\vartheta}}(\omega)(|t_1-t_2|^{\gamma_2}+|x_1-x_2|^{\gamma_1}),\\
t_i\in[\delta,T],\quad |x_i|\le M,\quad i=1,\ 2.
\end{split}
\end{equation}

In the sequel, we will use  that
\begin{equation}\label{eq:estpx}
\Bigl|\frac{\pr p}{\pr y} (t-s,x-y)\Bigr|\le  \frac{C_\lambda}{t-s}
 e^{-\frac{\lambda(x-y)^2}{t-s}}
\end{equation}
for any $\lambda\in (0,1/4)$ and some constant $C_\lambda$.

\section{Regularity property of the stochastic integral}
\label{sc:esti}

In this section, we obtain some properties of $\widetilde{\vartheta}$.

For any $j\in \rr$ and all $n\ge 0$, put
\[
d_{kn}^{(j)}=j+k 2^{-n},\quad 0\le k\le 2^n,\quad \Delta_{kn}^{(j)}=(d_{(k-1)n}^{(j)}, d_{kn}^{(j)}],\quad
1\le k\le 2^n\, .
\]

A following estimate is a key tool for the estimates of the stochastic integral.

\begin{lemma}\label{lm:essih}(Lemma~3.2~\cite{rads09})
Let $\mu$ be defined on $\bb(\rr)$, $Z$ be an arbitrary set, and $q(z, x): Z\times [j,
j+1]\to\rr$ be a function such that for some $1/2<\alpha<1$ and for each $z\in Z$ $q(z, \cdot)\in
B^\alpha_{22}\left([j, j+1]\right)$. Then the random function
\[
\eta(z)=\int_{[j, j+1]}q(z,  x)\,d\mu(x),\quad z\in Z,
\]
has a version $\tilde{\eta}(z)$ such that for some constant $C$ (independent of $z,\ j,\ \omega$) and each
$\omega\in\Omega$,
\begin{equation}
\label{eq:esqm}
\begin{split}
|\tilde{\eta}(z)|\le |q(z,  j)\mu([j, j+1])|
+ C \|q(z)\|_{B^\alpha_{22}([j, j+1])} \Bigl\{\sum_{n\ge 1}2^{n(1-2\alpha)}\sum_{1\le
k\le 2^{n}}|\mu(\Delta_{kn}^{(j)})|^2 \Bigr\}^{1/2}.
\end{split}
\end{equation}
\end{lemma}

Further, we study the properties of stochastic integral $\vartheta$.

\begin{lemma}\label{lm:thcont} Let Assumptions A\ref{assxs} and A\ref{assxm} hold. Then for version $\widetilde{\vartheta}$ that satisfies \eqref{eq:vervh}, we have
$\widetilde{\vartheta}\in \cc([0,T],\llf^2(\rr))$
\end{lemma}

\begin{proof}
To prove the continuity in $\llf^2(\rr)$, for fixed $\omega\in\Omega$ and arbitrary $t_n\to t_0$, consider
\begin{equation*}
\begin{split}
\|\widetilde{\vartheta}(t_n)-\widetilde{\vartheta}(t_0)\|_{\llf_2(\rr)}^2=\int_{\rr}|\widetilde{\vartheta}(t_n,x)-\widetilde{\vartheta}(t_0,x)|^2\,dx.
\end{split}
\end{equation*}
By~\eqref{eq:vervh}, $\widetilde{\vartheta}$ is continuous in both variables, therefore $\widetilde{\vartheta}(t_n,x)\to\widetilde{\vartheta}(t_0,x)$ for each $x$.

We will find $g\in\llf^2(\rr)$ such that
\begin{equation}\label{eq:thgest}
|\widetilde{\vartheta}(t,x)|=\Bigl|\int_{\rr} \int_0^t p(t-s,x-y) \sigma(s,y)\,ds\,d\mu(y)\Bigr|\le Cg(x),
\end{equation}
then the dominated convergence theorem implies our statement.

We apply Lemma~\ref{lm:essih} for
\[
q(z,y)= \int_0^t p(t-s,x-y) \sigma(s,y)\,ds,\quad z=(t,x)
\]
Our considerations will imply that the function $q(z,y)$ is continuous in~$y$ (see \eqref{eq:esio} and \eqref{eq:esdpss} below), thus the conditions of Lemma~\ref{lm:essih} are satisfied.

We fix $\omega\in\Omega$, $\theta>1$ and estimate
\begin{equation}\label{eq:sumtht}
\begin{split}
|\widetilde{\vartheta}(t,x)|^2=\Bigl|\int_{\rr} \int_0^t p(t-s,x-y) \sigma(s,y)\,ds\,d\mu(y)\Bigr|^2\\
=\Bigl|\sum_{j\in\zz}\int_{(j,j+1]} \int_0^t p(t-s,x-y) \sigma(s,y)\,ds\,d\mu(y)\Bigr|^2\\
\le \sum_{j\in\zz}(|j|+1)^{-\theta} \sum_{j\in\zz}(|j|+1)^{\theta} \Bigl|\int_{(j,j+1]} \int_0^t p(t-s,x-y) \sigma(s,y)\,ds\,d\mu(y)\Bigr|^2\\
=C \sum_{j\in\zz}(|j|+1)^{\theta} \Bigl|\int_{(j,j+1]}  q(t,x,y) \,d\mu(y)\Bigr|^2\\
\stackrel{\eqref{eq:esqm}}{\le}C\sum_{j\in\zz}(|j|+1)^{\theta} \Bigl(|q(t,x,j)\mu((j, j+1])|^2
+ \|q(t)\|^2_{B^{2,2}_\alpha([j,j+1])}
\Bigl\{\sum_{n\ge 1}2^{-n(2\alpha-1)}\sum_{1\le k\le 2^{n}}|\mu(\Delta_{kn}^{(j)})|^2\Bigr\}\Bigr).
\end{split}
\end{equation}

Below we will use the following simple estimates:
\begin{equation}
\begin{split}
\int_0^t\frac1r e^{-\frac{b}{r}} dr\stackrel{{b}/{r} =z}{=}\int_{b/t}^{\infty}\frac{1}{z} e^{-z}\, dz
  \le  \ii_{\{b\ge t\}}\int_{b/t}^{\infty} e^{-z}\,dz+\ii_{\{b<t\}}\Bigl(\int_{b/t}^1\frac{1}{z}\, dz+\int_{1}^{\infty} e^{-z}\,dz\Bigr)\\
	\le  \ii_{\{b\ge t\}}e^{-{b/t}}+\ii_{\{b<t\}} \Bigl(\ln{\frac{t}{b}}+1\Bigr),\label{eq:estb}
\end{split}
\end{equation}
\begin{equation}\label{eq:estbss}
\begin{split}
\int_0^t\frac{1}{\sqrt{t-s}} e^{-\frac{b}{t-s}} ds\le \int_0^t\frac{1}{\sqrt{t-s}} e^{-\frac{b}{T}} ds\le C e^{-\frac{b}{T}}.
\end{split}
\end{equation}

We have that
\begin{equation}\label{eq:estqv}
\begin{split}
|q(t,x,y)|=\Bigl|\int_0^t p(t-s,x-y) \sigma(s,y)\,ds\Bigr|
\le C_\sigma C\int_0^t \dfrac{e^{-\frac{(x-y)^2}{4(t-s)}}}{\sqrt{t-s}}\,ds
\stackrel{\eqref{eq:estbss}}{\le} C e^{-\frac{(x-y)^2}{4T}}.
\end{split}
\end{equation}

To estimate $\|q(t)\|_{B^{2,2}_\alpha([j,j+1])}$, we consider
\begin{equation*}
\begin{split}
q(t,x,y+h)-q(t,x,y)  = \int_0^t p(t-s,x-y)(\sigma(s,y+h)-\sigma(s,y))\,ds\\
 + \int_0^t(p(t-s,x-y-h)-p(t-s,x-y))\sigma(s,y+h)\,ds
:=I_1+I_2.
\end{split}
\end{equation*}

For $I_1$, we obtain
\begin{equation}\label{eq:esio}
\begin{split}
|I_1|\le L_{\sigma}h^{\beta(\sigma)} \int_0^t p(t-s,x-y)\,ds=C h^{\beta(\sigma)} \int_0^t \dfrac{e^{-\frac{(x-y)^2}{4(t-s)}}}{\sqrt{t-s}}\,ds
\stackrel{\eqref{eq:estbss}}{\le} C h^{\beta(\sigma)} e^{-\frac{(x-y)^2}{4T}}.
\end{split}
\end{equation}

For $I_2$, assuming $0\le h\le 1$,  we get
\begin{equation}\label{eq:esdpss}
\begin{split}
 |I_2|\stackrel{\rm A\ref{assxs}}{\le}  C_{\sigma}\int_0^t|p(t-s,x-y-h)-p(t-s,x-y)|\,ds\\
 = C_{\sigma}\int_0^t \Bigl| \int_{x-y-h}^{x-y} \dfrac{\partial p(t-s,v)}{\partial v}\,dv\Bigr|\,ds\\
 \stackrel{\eqref{eq:estpx}}{\le}  C_{\sigma}C_{\lambda}\int_0^t\frac{1}{t-s}\int_{x-y-h}^{x-y} e^{-\frac{\lambda v^2}{t-s}}\,dv\,ds\\
  \stackrel{\eqref{eq:estb}}{\le}   C\int_{x-y-h}^{x-y} e^{-\frac{\lambda v^2 }{t}} \ii_{ \{\lambda v^2\ge t\} }\,dv
+C\int_{x-y-h}^{x-y}\Bigl(\Bigl|\ln{\frac{T}{\lambda v^2}}\Bigr|+1\Bigr)\ii_{ \{\lambda v^2< t\} }dv\\
  {\le} C\int_{x-y-h}^{x-y} e^{-\frac{\lambda v^2}{t}} \ii_{ \{\lambda v^2\ge t\} }\,dv +C\int_{x-y-h}^{x-y}|\ln{|v|}|\ii_{ \{\lambda v^2< t\} }dv\\
 \stackrel{(*)}{\le}   C\int_{x-y-h}^{x-y} e^{-\frac{\lambda v^2}{t}} \ii_{ \{\lambda v^2\ge t\} }\,dv +C\ii_{ \{|x-y|< \sqrt{t/\lambda}+1\}}\int_0^{h/2}|\ln z|\,dz\\
 \stackrel{(**)}{\le}  Ch e^{-\frac{\tilde{\lambda} (x-y)^2}{T}}+C\ii_{ \{|x-y|< \sqrt{t/\lambda}+1\}}(z-z\ln z)\Big|_0^{h/2}\\
\le Ch e^{-\frac{\tilde{\lambda} (x-y)^2}{T}} +Ch |\ln h|\ii_{ \{|x-y|< \sqrt{t/\lambda}+1\}} \\
\le Ch e^{-\frac{\tilde{\lambda} (x-y)^2 }{T}} +C h^{\beta(\sigma)}\ii_{ \{|x-y|< \sqrt{T/\lambda}+1\}}
\le C h^{\beta(\sigma)} e^{-\frac{\tilde{\lambda} (x-y)^2}{T}} .
\end{split}
\end{equation}
In (*) we have used that the maximal value of $\int_{x_1}^{x_2}|\ln{|v|}|\ii_{ \{|v|< C\} }dv$ for fixed $|x_1-x_2|$ is achieved when $x_2=-x_1$.

In (**) we applied that for $0<\tilde{\lambda}<\lambda$, $x-y-h\le v\le x-y$ holds
\begin{equation*}
\begin{split}
   Ce^{-\frac{\lambda v^2 }{T}}\le e^{-\frac{\tilde{\lambda} (x-y)^2 }{T}}.
\end{split}
\end{equation*}

Further, we note that  for $y\in[j,j+1]$
\begin{equation*}
  e^{-\frac{(x-y)^2\tilde{\lambda}}{T}}\le   e^{\frac{(1-(|x-j|-1)^2)\tilde{\lambda}}{T}},
\end{equation*}
and \eqref{eq:esio} and \eqref{eq:esdpss} imply that
\begin{equation}\label{eq:estqd}
|q(t,x,y+h)-q(t,x,y)|\le   C h^{\beta(\sigma)} e^{\frac{(1-(|x-j|-1)^2)\tilde{\lambda}}{T}}.
\end{equation}

Now, from \eqref{eq:bssn}, \eqref{eq:estqv}, and~\eqref{eq:estqd}, for some $1/2<\alpha<\beta(\sigma)$ we get that
\begin{equation}\label{eq:estqb}
   \|q(t)\|_{B^{2,2}_\alpha([j,j+1])}\le C e^{\frac{(1-(|x-j|-1)^2)\tilde{\lambda}}{T}}.
\end{equation}

From \eqref{eq:sumtht}, \eqref{eq:estqv}, and \eqref{eq:estqb} we obtain~\eqref{eq:thgest}, where
\begin{equation}\label{eq:defg}
\begin{split}
g^2(x)=\sum_{j\in\zz}(|j|+1)^{\theta} \Bigl(e^{\frac{2(1-(|x-j|-1)^2)\tilde{\lambda}}{T}}\mu^2((j, j+1])\\
+ e^{\frac{2(1-(|x-j|-1)^2)\tilde{\lambda}}{T}}
\Bigl\{\sum_{n\ge 1}2^{-n(2\alpha-1)}\sum_{1\le k\le 2^{n}}|\mu(\Delta_{kn}^{(j)})|^2\Bigr\}\Bigr).
\end{split}
\end{equation}

We will check that $g\in\llf^2(\rr)$, and get
\begin{equation}\label{eq:estig}
\begin{split}
\int_{\rr}g^2(x)\,dx{\le}  C\sum_{j\in\zz}(|j|+1)^{\theta} \mu^2((j, j+1])\\
+C\sum_{j\in\zz}(|j|+1)^{\theta} \Bigl\{\sum_{n\ge 1}2^{-n(2\alpha-1)}\sum_{1\le k\le 2^{n}}|\mu(\Delta_{kn}^{(j)})|^2\Bigr\}.
\end{split}
\end{equation}

Here the sums with $\mu$ have a form
 $\sum_{l=1}^{\infty}\Bigl(\int_{\rr} f_l\,d\mu\Bigr)^2$, where
\begin{eqnarray*}
\{f_l (y),\ l\ge 1\}  =  \{(|j|+1)^{\theta/2}\, \ii_{(j, j+1]}(y),\ j\in \zz\},\\
\{f_l (y),\ l\ge 1\}  =  \{(|j|+1)^{\theta/2} 2^{-n(\alpha-1/2)}
\ii_{\Delta_{kn}^{(j)}}(y),\ j\in \zz,\ n\ge 1,\ 1\le k\le 2^{n}\}.
\end{eqnarray*}
We have
\[
\sum_{l=1}^{\infty}|f_l(y)|\le C(|y|+1)^{\theta/2}.
\]
By A\ref{assxm}, $\sum_{l=1}^{\infty}|f_l|$ is integrable w.r.t. $\mu$ on $\rr$ in both cases. From Lemma~\ref{lm:fkmu} it follows that for
\begin{equation}\label{eq:omta}
\Omega_{\theta,\alpha}=\Bigl\{\omega\in\Omega:\ \sum_{l=1}^{\infty}\Bigl(\int_{\xx} f_l\,d\mu \Bigr)^2<+\infty\Bigr\}
\end{equation}
holds $\pp(\Omega_{\theta,\alpha}=1)$. For each $\omega\in\Omega_{\theta,\alpha}$, $g(\cdot,\omega)\in\llf^2(\rr)$, and $\widetilde{\vartheta}(\cdot,\omega)\in \cc([0,T],\llf^2(\rr))$.
\end{proof}

In particular, Lemma~\ref{lm:thcont} implies that for each $\omega\in\Omega_{\theta,\alpha}$ holds
\begin{equation}\label{eq:thnbound}
\sup_{t\in[0,T]} \|\widetilde{\vartheta}(t)\|_{\llf^2(\rr)}<\infty.
\end{equation}

We will need one more boundness result about $\widetilde{\vartheta}$.

\begin{lemma}\label{lm:thbound} Let Assumptions A\ref{assxs} and A\ref{assxm} hold. Then, for version $\widetilde{\vartheta}$ that satisfies \eqref{eq:vervh}, we have
\[
\sup_{t\in [0,T], x\in\rr}|\widetilde{\vartheta}(t,x)|<\infty\ {\rm a.s.}
\]
\end{lemma}

\begin{proof} In the proof of Lemma~\ref{lm:thcont}, for function $g$ defined in~\eqref{eq:defg}, we obtained that estimate~\eqref{eq:thgest} holds. From \eqref{eq:defg} it follows that
\begin{equation*}
\begin{split}
g^2(x){\le}  C\sum_{j\in\zz}(|j|+1)^{\theta} \mu^2((j, j+1])
+C\sum_{j\in\zz}(|j|+1)^{\theta} \Bigl\{\sum_{n\ge 1}2^{-n(2\alpha-1)}\sum_{1\le k\le 2^{n}}|\mu(\Delta_{kn}^{(j)})|^2\Bigr\}.
\end{split}
\end{equation*}
Further, we repeat the proof of Lemma~\ref{lm:thcont} after \eqref{eq:estig}, and for each $\omega\in\Omega_{\theta,\alpha}$ obtain that $\sup_{x\in\rr}|g(x,\omega)|<\infty$.
\end{proof}

\section{Solution to the auxiliary equation}
\label{sc:burgaux}

Consider the operator $\pi_N:\llf^2(\rr)\to \llf^2(\rr)$ such that
\begin{equation*}
\pi_N(v)=\left\{
\begin{array}{ll}
  v,\quad & \|v\|_{\llf^2(\rr)}\le N,\\
  N\dfrac{v}{\|v\|_{\llf^2(\rr)}}, \quad & \|v\|_{\llf^2(\rr)}>N.
\end{array}
\right.
\end{equation*}
For Hilbert space $\llf^2(\rr)$, it is easy to check that
\begin{equation}\label{eq:difpin}
\|\pi_N(v)-\pi_N(w)\|_{\llf^2(\rr)}\le \|v-w\|_{\llf^2(\rr)}.
\end{equation}

In this section, we study the auxiliary equation
\begin{equation}\label{eq:bemfpiv}
\begin{split}
u(t,x)=\int_{\rr} p(t,x-y)u_0(y)\,dy
+\int_0^t \int_{\rr}p(t-s,x-y) f(s,y,\pi_N(u)(s,y))\,dy\,ds\\
-\int_0^t \int_{\rr}\frac{\pr p}{\pr y}(t-s,x-y) g(s,y,\pi_N(u)(s,y))\,dy\,ds
+\widetilde{\vartheta}(t,x).
\end{split}
\end{equation}
We consider~\eqref{eq:bemfpiv} for fixed  $\omega\in\Omega_{\theta,\alpha}$ (see~\eqref{eq:omta}), and will prove the existence and uniqueness of the solution.

Let us consider the linear operators
\begin{eqnarray*}
(J_1 v)(t,x)=\int_0^t \int_{\rr} p(t-s,x-y) v(s,y)\,dy\,ds,\\
(J_2 w)(t,x)=\int_0^t \int_{\rr} \dfrac{\pr p}{\pr y}(t-s,x-y) w(s,y)\,dy\,ds,
\end{eqnarray*}
where $t\in [0,T]$, $x\in\rr$, $v,w\in \llf^{\infty}([0,T],\llf^2(\rr))$.

We recall some lemmas from~\cite{gyonua}.

\begin{lemma} (Lemma 3.1~\cite{gyonua}) The operator $J_1$ is bounded from $\llf^2([0,T],\llf^2(\rr))$ to $\cc([0,T],\llf^2(\rr))$, and the following estimates hold:
\begin{equation}
\label{eq:estj1d}
\|J_1 v(t)-J_1 v(r)\|_{\llf^2(\rr)}\le C|t-r|^{1/3}\Bigl(\int_0^t 	\|v(s)\|_{\llf^2(\rr)}^2\,ds\Bigr)^{1/2},
\end{equation}
where $0\le r, t\le T$.
\end{lemma}

\begin{lemma} (Lemma 3.2~\cite{gyonua})
The operator $J_2$ is bounded from $\llf^5([0,T],\llf^1(\rr))$ to $\cc([0,T],\llf^2(\rr))$, and the following estimates hold:
\begin{eqnarray}\label{eq:estj2}
\|J_2 w(t)\|_{\llf^2(\rr)}\le C\int_0^t (t-s)^{-3/4}\|w(s)\|_{\llf^1(\rr)}\,ds,\\
\label{eq:estj2d}
\|J_2 w(t)-J_2 w(r)\|_{\llf^2(\rr)}\le C|t-r|^{1/21}\Bigl(\int_0^t 	\|w(s)\|_{\llf^1(\rr)}^5\,ds\Bigr)^{1/5}.
\end{eqnarray}
where $0\le r, t\le T$.
\end{lemma}

The following statement is the analogue of Proposition~4.1~\cite{gyonua}.

\begin{lemma} Let A\ref{assxfall} and A\ref{assxgall} hold, and
\[
u_0\in\llf^2(\rr),\quad |\widetilde{\vartheta}(t,x)|\le C(\omega),\quad \|\widetilde{\vartheta}(t)\|_{\llf^2(\rr)}\le C(\omega)
\]
for some finite constant $C(\omega)$. Then for any fixed $N>0$ and $\omega\in\Omega_{\theta,\alpha}$ there equation \eqref{eq:bemfpiv} has a unique solution $u\in\cc([0,T],\llf^2(\rr))$.
\end{lemma}

\begin{proof}

\emph{Step 1 (existence and uniqueness of the solution).}

Denote
\begin{equation*}
\begin{split}
(\aaa_1 u)(t,x)=\int_0^t \int_{\rr}p(t-s,x-y) f(s,y,\pi_N(u)(s,y))\,dy\,ds,\\
(\aaa_2 u)(t,x)=-\int_0^t \int_{\rr}\frac{\pr p}{\pr y}(t-s,x-y) g(s,y,\pi_N(u)(s,y))\,dy\,ds.
\end{split}
\end{equation*}

Fix $\lambda>0$. Let $\hh$ denote the Banach space of $\llf^2(\rr)$-valued functions $u(t,x)$ such that $u(0,x)=u_0(x)$, with the norm
\[
\|u(t)\|^2_{\hh}=\int_0^T e^{-\lambda t}\|u(t)\|_{\llf^2(\rr)}^2\,dt.
\]
Define the operator $\aaa:\hh\to\hh$ such that
\begin{equation}\label{eq:operaa}
(\aaa u)(t,x)=\int_{\rr} p(t,x-y)u_0(y)\,dy+(\aaa_1 u)(t,x)+(\aaa_2 u)(t,x)+\widetilde{\vartheta}(t,x).
\end{equation}
We will prove that operator $\aaa$ is a contraction for large $\lambda$.

Using the inequality
\[
\Bigl\|\int_0^t f(s)\,ds\Bigr\|\le \int_0^t \|f(s)\|\,ds,
\]
we have that
\begin{equation}\label{eq:aauaav}
\begin{split}
\|(\aaa_1 u)(t)-(\aaa_1 v)(t)\|_{\llf^2(\rr)}\\
\le \int_0^t \Bigl\|\int_{\rr}p(t-s,x-y) (f(s,y,\pi_N(u)(s,y))-f(s,y,\pi_N(v)(s,y)))\,dy\Bigr\|_{\llf^2(\rr)}\,ds\\
\stackrel{A\ref{assxfall}}{\le}\int_0^t \Bigl\|\int_{\rr}p(t-s,x-y)(a_2(y)+L|\pi_N(u)(s,y)|+L|\pi_N(v)(s,y)|)\\
\times |\pi_N(u)(s,y)-\pi_N(v)(s,y)|\,dy\Bigr\|_{\llf^2(\rr)}\,ds\\
\stackrel{(**)}{\le}\int_0^t \|p(t-s,x-y)\|_{\llf^2(\rr)} \|(a_2(y)+L|\pi_N(u)(s,y)|+L|\pi_N(v)(s,y)|)\\
\times|\pi_N(u)(s,y)-\pi_N(v)(s,y)|\|_{\llf^1(\rr)}\,ds.
\end{split}
\end{equation}
Here in (**) we used the  inequality for convolution $\|v*w\|_{\llf^2(\rr)}\le \|v\|_{\llf^2(\rr)}\|w\|_{\llf^1(\rr)}$. Further, we have
\begin{equation*}
\begin{split}
\|p(t-s,x-y)\|^2_{\llf^2(\rr)}=C (t-s)^{-1/2},\\
\|(a_2(y)+L|\pi_N(u)(s,y)|+L|\pi_N(v)(s,y)|)|\pi_N(u)(s,y)-\pi_N(v)(s,y)|\|_{\llf^1(\rr)}\\
{\le}\|(a_2(y)+L|\pi_N(u)(s,y)|+L|\pi_N(v)(s,y)|\|_{\llf^2(\rr)} \|\pi_N(u)(s,y)-\pi_N(v)(s,y)\|_{\llf^2(\rr)}\\
\stackrel{\eqref{eq:difpin}}{\le} (\|a_2\|_{\llf^2(\rr)}+2LN) \|u(s)-v(s)\|_{\llf^2(\rr)}.
\end{split}
\end{equation*}

Applying the H\"{o}lder inequality, we get
\begin{equation}\label{eq:estao}
\begin{split}
\|(\aaa_1 u)(t)-(\aaa_1 v)(t)\|_{\llf^2(\rr)}^2
\le C_N \Bigl(\int_0^t  (t-s)^{-1/4} \|u(s)-v(s)\|_{\llf^2(\rr)}\,ds\Bigr)^2\\
\le C_N \int_0^t  (t-s)^{-3/4} \|u(s)-v(s)\|_{\llf^2(\rr)}^2\,ds,
\end{split}
\end{equation}
where $C_N$ denotes some constants that may depend on $N$.

For $\aaa_2$, we have
\begin{equation}\label{eq:estud}
\begin{split}
\|(\aaa_2 u)(t)-(\aaa_2 v)(t)\|_{\llf^2(\rr)}^2\\
=\Bigl\|\int_0^t \int_{\rr}\frac{\pr p}{\pr y}(t-s,x-y) (g(s,y,\pi_N(u)(s,y))
-g(s,y,\pi_N(v)(s,y)))\,dy\,ds\Bigr\|_{\llf^2(\rr)}^2\\
\stackrel{\eqref{eq:estj2}}{\le}\Bigl(\int_0^t (t-s)^{-3/4}\|(b_3(y)+L|\pi_N(u)(s,y)|
+L|\pi_N(v)(s,y)|)|\pi_N(u)(s,y)-\pi_N(v)(s,y)|\|_{\llf^1(\rr)}\,ds\Bigr)^2\\
{\le}C\int_0^t (t-s)^{-3/4}\|(b_3(y)+L|\pi_N(u)(s,y)|
+L|\pi_N(v)(s,y)|)|\pi_N(u)(s,y)-\pi_N(v)(s,y)|\|_{\llf^1(\rr)}^2\,ds\\
{\le}C\int_0^t (t-s)^{-3/4}\|b_3(y)+L|\pi_N(u)(s,y)|+L|\pi_N(v)(s,y)|\|_{\llf^2(\rr)}^2\|\pi_N(u)(s,y)-\pi_N(v)(s,y)\|_{\llf^2(\rr)}^2\,ds\\
\le C\int_0^t (t-s)^{-3/4}(\|b_3\|_{\llf^2(\rr)}+2LN)^2\|u(s,y)-v(s,y)\|_{\llf^2(\rr)}^2\,ds\\
=C_N \int_0^t (t-s)^{-3/4}\|u(s)-v(s)\|_{\llf^2(\rr)}^2\,ds.
\end{split}
\end{equation}

Therefore,
\begin{equation*}
\begin{split}
\|\aaa u-\aaa v\|_{\hh}^2= \int_0^T e^{-\lambda t}\|(\aaa u)(t)-(\aaa v)(t)\|^2_{\llf^2(\rr)}\,dt\\
\stackrel{\eqref{eq:estao},\eqref{eq:estud}}{\le} C_N \int_0^T  e^{-\lambda t} \int_0^t (t-s)^{-3/4}\|u(s)-v(s)\|_{\llf^2(\rr)}^2\,ds \,dt\\
= C_N \int_0^T  e^{-\lambda s} \|u(s)-v(s)\|_{\llf^2(\rr)}^2 \int_s^T (t-s)^{-3/4}e^{-\lambda (t-s)} \,dt \,ds\\
\le C_N \int_0^T  e^{-\lambda s} \|u(s)-v(s)\|_{\llf^2(\rr)}^2 \int_0^\infty r^{-3/4}e^{-\lambda r} \,dt \,dr\\
= C_N \|u-v\|_{\hh}^2 \int_0^\infty r^{-3/4}e^{-\lambda r} \,dt ,
\end{split}
\end{equation*}
and for large $\lambda$ we can get
\[
C_N\int_0^\infty r^{-3/4}e^{-\lambda r} \,dt <1.
\]
Then the operator $\aaa$ on $\hh$ is a contraction and has a unique fixed point that is the solution of~\eqref{eq:bemfpiv}.

\emph{Step 2 (continuity of the solution in $\llf^2(\rr)$).}

We will demonstrate that $u\in\cc([0,T],\llf^2(\rr))$ provided that $\aaa u=u$. We consider each term in~\eqref{eq:operaa}, and obtain
\begin{equation*}
\begin{split}
\|(\aaa_1 u)(t)-(\aaa_1 u)(r)\|^2_{\llf^2(\rr)}
\stackrel{\eqref{eq:estj1d}}{\le}(t-r)^{2/3}\int_0^t \|f(s,y,\pi_N(v)(s,y))\|^2_{\llf^2(\rr)}\,ds\\
\stackrel{A\ref{assxfall}}{\le} (t-r)^{2/3} 2T(\|a_1\|^2_{\llf^2(\rr)}+K^2N^2),\\
\|(\aaa_2 u)(t)-(\aaa_2 u)(r)\|^5_{\llf^2(\rr)}\stackrel{\eqref{eq:estj2d}}{\le}(t-r)^{5/21}\int_0^t \|g(s,y,\pi_N(v)(s,y))\|^5_{\llf^1(\rr)}\,ds\\
\stackrel{A\ref{assxgall}}{\le} (t-r)^{5/21}\cdot CT(\|b_1\|^5_{\llf^1(\rr)}+\|b_2\pi_N(v)\|^5_{\llf^1(\rr)}+K^5\|\pi_N^2(v)\|^5_{\llf^1(\rr)})\\
\le C (t-r)^{5/21} (\|b_1\|^5_{\llf^1(\rr)}+\|b_2\|^5_{\llf^2(\rr)}\|\pi_N(v)\|^5_{\llf^2(\rr)}+K^5\|\pi_N(v)\|^5_{\llf^2(\rr)})\\
\le C (t-r)^{5/21} (\|b_1\|^5_{\llf^1(\rr)}+\|b_2\|^5_{\llf^2(\rr)}N^5+K^5N^5),
\end{split}
\end{equation*}
therefore $(\aaa_1 u)(t)$ and $(\aaa_2 u)(t)$ are continuous. For $\tilde{\vartheta}$, we refer to Lemma~\ref{lm:thcont}, convolution $p*u_0$ is continuous by standard properties of the heat semigroup.
\end{proof}

\section{Solution to the main equation}
\label{sc:burgsol}

We will use one more statement from~\cite{gyonua}.

\begin{lemma} (Lemma~4.2~\cite{gyonua})\label{lm:estuz}
Let $\zeta=\{\zeta(t,x),t\in[0,T],x\in\rr\}$ be a continuous and bounded function belonging to $\cc([0,T],\llf^2(\rr))$. Let $v\in\cc([0,T],\llf^2(\rr))$ be a solution of the integral equation
\begin{equation}\label{eq:bemfgv}
\begin{split}
v(t,x)=\int_{\rr} p(t,x-y)u_0(y)\,dy
+\int_0^t \int_{\rr}p(t-s,x-y) f(s,y,v(s,y)+\zeta(s,y))\,dy\,ds\\
-\int_0^t \int_{\rr}\frac{\pr p}{\pr y}(t-s,x-y) g(s,y,v(s,y)+\zeta(s,y))\,dy\,ds,
\end{split}
\end{equation}
where $u_0(y)\in\llf^2(\rr)$, and $f$ and $g$ satisfy the assumptions A\ref{assxfall} and A\ref{assxgall}.

Then we have
\begin{equation}\label{eq:estvn}
\|v(t)\|^2_{\llf^2(\rr)}\le (\|u_0\|^2_{\llf^2(\rr)}+C_1(1+R_1(\zeta)))\exp\{C_2(1+R_2(\zeta))\},
\end{equation}
where the constants $C_1$ and $C_2$ depend only on $T$ and on the functions $a_i,b_i$ and the constants $K$ and $L$ appearing in the hypothesis  A\ref{assxfall} and A\ref{assxgall}, and
\begin{equation*}
\begin{split}
R_1(\zeta)=\sup_{s\in [0,T]}(\|\zeta(s)\|^2_{\llf^2(\rr)}+\|\zeta(s)\|^4_{\llf^4(\rr)}+\|\zeta(s)\|^2_{\llf^{\infty}(\rr)}),\quad
R_2(\zeta)=\sup_{s\in [0,T]}\|\zeta(s)\|^2_{\llf^{\infty}(\rr)}.
_{}\end{split}
\end{equation*}
\end{lemma}

The main result of the paper is the following.

\begin{thm}\label{th:burgers} Let Assumptions A\ref{assxu} -- A\ref{assxm} hold. Then equation~\eqref{eq:bemfg} has a unique solution which is continuous with values in $\llf^2(\rr)$.
\end{thm}

\begin{proof}
We consider equation~\eqref{eq:bemfg} for each fixed $\omega\in\Omega_{\theta,\alpha}$, take the version $\widetilde{\vartheta}$ that satisfies \eqref{eq:vervh}.

We apply Lemma~\ref{lm:estuz} to
\[
\zeta(t,x)=\int_{\rr} \int_0^t p(t-s,x-y) \sigma(s,y)\,ds\,d\mu(y)=\widetilde{\vartheta}(t,x),
\]
assumptions on $\zeta$ are fulfilled by Lemmas \ref{lm:thcont} and \ref{lm:thbound}.

For given $\zeta$, set
\begin{equation}\label{eq:defn}
N=N_1+\sup_{t\in [0,T]} \|\zeta(t)\|_{\llf^{2}(\rr)}+1,
\end{equation}
where $N_1^2$ equal to the right hand side of~\eqref{eq:estvn}.
For given $N$, take $u$ that is a solution of \eqref{eq:bemfpiv}.

Denote
\[
t_N=\sup \{t: \|u(t)\|_{\llf^{2}(\rr)}\le N\},
\]
for $t\le t_N$ we have $\pi_N(u)=u$ and $v=u-\zeta$ satisfies~\eqref{eq:bemfgv}.

By Lemma~\ref{lm:estuz}, $\|v(t)\|_{\llf^{2}(\rr)}\le N_1$, and
\[
\|u(t)\|_{\llf^{2}(\rr)}\le \|v(t)\|_{\llf^{2}(\rr)}+\|\zeta(t)\|_{\llf^{2}(\rr)}\le N-1,\quad t\le t_N.
\]
Therefore, $t_N=T$, and never we get $\|u(t)\|_{\llf^{2}(\rr)}> N$. The solution of \eqref{eq:bemfpiv} will satisfy \eqref{eq:bemfg}.

Conversely, in \eqref{eq:bemfg} we have that $\|u_0\|_{\llf^{2}(\rr)}\le N-1$, up to the moment $t_N=T$ equation~\eqref{eq:bemfg} coinsides with \eqref{eq:bemfpiv} and has a unique solution.
\end{proof}

\section{Averaging principle}
\label{sc:aver}

In this section, we consider the equation~\eqref{eq:burgg}, where functions $f$ and $g$ do not depend on the time variable $t$ and study the averaging of the stochastic term.

For $\varepsilon>0$, consider equation
\begin{equation}\label{eq:burgav}
\begin{split}
\frac{\pr u_\varepsilon}{\pr t}=\frac{\pr^2 u}{\pr x^2}+f(x,u_\varepsilon(t,x))+\frac{\pr g}{\pr x}(x,u_\varepsilon(t,x)) + \sigma(t/\varepsilon, x)\frac{\pr  \mu}{\pr x},\quad
u(0,x)=u_0(x).
\end{split}
\end{equation}

Assume that for each $y\in\rr$ there exists the following limit
\begin{equation}\label{eq:barf}
\bar{\sigma}(y)=\lim_{t\to\infty}\frac{1}{t}\int_0^t \sigma(s,y)\,ds.
\end{equation}
It is easy to see that $\bar{\sigma}(y)$ satisfies Assumption~A\ref{assxs} with the same constants $C_\sigma$, $L_\sigma$.

We will study convergence $u_{\varepsilon}(t, x)\to\bar{u}(t, x),\  \varepsilon\to 0$,
were $\bar{u}$ is the solution of the averaged equation
\begin{equation}
\label{eq:hshfav}
\begin{split}
\frac{\pr \bar{u}}{\pr t}=\frac{\pr^2 \bar{u}}{\pr x^2}+f(x,\bar{u}(t,x))+\frac{\pr g}{\pr x}(x,\bar{u}(t,x)) + \bar\sigma(x)\frac{\pr  \mu}{\pr x},\quad
\bar{u}(0,x)=u_0(x).
\end{split}
\end{equation}

The mild forms of \eqref{eq:burgav} and \eqref{eq:hshfav} are respectively
\begin{equation*}
\begin{split}
u_\varepsilon(t,x)=\int_0^t p(t,x-y)u_0(y)\,dy
+\int_0^t \int_{\rr}p(t-s,x-y) f(y,u_\varepsilon(s,y))\,dy\,ds\\ -
\int_0^t \int_{\rr}\frac{\pr }{\pr y}p(t-s,x-y) g(y,u_\varepsilon(s,y))\,dy\,ds
+\int_{\rr}d\mu(y) \int_0^t p(t-s,x-y) \sigma(s/\varepsilon,y)\,ds.
\end{split}
\end{equation*}
and
\begin{equation*}
\begin{split}
\bar{u}(t, x) = \int_{\rr}{p}(t, x-y)u_0(y)\,dy
+ \int_0^t ds \int_{\rr}{p}(t-s, x-y){f}(y, \bar{u}(s, y))\,dy\\
-\int_0^t \int_{\rr}\frac{\pr }{\pr y}p(t-s,x-y) g(y,\bar{u}(s,y))\,dy\,ds
+ \int_{\rr} d\mu(y) \int_0^t {p}(t-s, x-y)\bar{\sigma}(y)\,ds\, .
\end{split}
\end{equation*}

We also impose the following additional condition that is standard in the averaging principle.

\begin{assx}\label{assavg}
The function $G_{\sigma}(r,y)=\int_0^{r} (\sigma(s,y)-\bar{\sigma}(y))\,ds,\ r\in\rr_+,\ y\in\rr$
is bounded.
\end{assx}

This holds, for example, if $\sigma(s,y)$ is bounded and periodic in $s$ for each fixed $y$, and the set of values of minimal periods is bounded. Obviously, A\ref{assavg} implies~\eqref{eq:barf}.

\begin{thm}\label{th:averh} Assume that Assumptions A\ref{assxu}--A\ref{assavg} hold. Then there exists a versions of  $u_{\varepsilon}$ and $\bar{u}$ such that
\begin{equation}\label{eq:ueut}
\sup_{t\in [0,T]}\|u_{\varepsilon}(t)-\bar{u}(t)\|_{\llf^2(\rr)}\to 0\ a.~s.
\end{equation}
\end{thm}

\begin{proof}
In Step 1 of the proof of  Theorem~1~\cite{radavsh19} (or Theorem 7.1~\cite{radbook}), for the stochastic integrals in
\begin{eqnarray*}
\xi_{\varepsilon} (t,x)= \int_{\rr} d\mu(y) \int_0^t {p}(t-s, x-y)\sigma(s/\varepsilon, y)ds
- \int_{\rr} d\mu(y) \int_0^t {p}(t-s, x-y)\bar{\sigma}(y)ds
\end{eqnarray*}
and some $\gamma_1>0$ it was proved that there exists a version such that
\begin{equation*}
|{\xi_\varepsilon}(t,x)|\le C(\omega){\varepsilon}^{\gamma_1}\ \textrm{a.~s.}
\end{equation*}
for all $\omega\in\Omega_1$, $\pp(\Omega_1)=1$.

As in the proof of Lemma~\ref{lm:thcont}, for function $g$ defined in~\eqref{eq:defg}, we obtain that
\begin{equation*}
|{\xi_\varepsilon}(t,x)|\le Cg(x)\ \textrm{a.~s.},
\end{equation*}
where $g\in \llf^2(\rr)$ is independent of $\varepsilon$.

The dominated convergence theorem imply that for each $t\in [0,T]$ and each $\omega\in\Omega_{\theta,\alpha}\cap \Omega_1$ holds
\begin{equation}\label{eq:esxigl}
\|{\xi_\varepsilon}(t)\|_{\llf^2(\rr)}\to 0,\quad \varepsilon\to 0.
\end{equation}

In the proof of Theorem~\ref{th:burgers}, it was obtained that
\begin{equation}\label{eq:estut}
\|u(t)\|_{\llf^2(\rr)}\le N,
\end{equation}
where $N$ is defined in~\eqref{eq:defn}.

We can see that also
\begin{equation}\label{eq:estute}
\|u_{\varepsilon}(t)\|_{\llf^2(\rr)}\le N
\end{equation}
for all $\varepsilon>0$ for the same $N$. To explain this, note that in~\eqref{eq:defn} $N_1$ depends only on $T$, on the functions $a_i,b_i$ and the constants $K$ and $L$ appearing in the assumptions A\ref{assxfall} and A\ref{assxgall}. For $\zeta(t,s)=\vartheta(t,s)$, in the proof of Lemma~\ref{lm:thcont} we obtained that
\[
\sup_{t\in[0,T]} \|\vartheta(t)\|^2_{\llf^{2}(\rr)}\le C\int_{\rr} g^2(x)\,dx,
\]
where the right-hand side may be estimated by \eqref{eq:estig}. The constants in \eqref{eq:estig} may depend on $C_{\sigma}$ and $L_{\sigma}$ from assumption A\ref{assxs}, but are independent of $\varepsilon$.

Further, we obtain
\begin{equation*}
\begin{split}
\|u_{\varepsilon}(t)- \bar{u}(t)\|_{\llf^2(\rr)}
\le \Bigl\|\int_0^t ds \int_{\rr}{p}(t-s, x-y)(f(y,u_{\varepsilon}(s, y))-f(y, \bar{u}(s, y))\,dy\Bigr\|_{\llf^2(\rr)}\\
+\Bigl\|\int_0^t ds \int_{\rr}\dfrac{\pr }{\pr y}{p}(t-s, x-y)(g(y,u_{\varepsilon}(s, y))-g(y, \bar{u}(s, y))\,dy\Bigr\|_{\llf^2(\rr)}\\
+ \|{\xi_\varepsilon}(t)\|_{\llf^2(\rr)}:=J_1+J_2+\|{\xi_\varepsilon}(t)\|_{\llf^2(\rr)}.
\end{split}
\end{equation*}
For $J_1$, as in \eqref{eq:aauaav}--\eqref{eq:estao}, taking into account~\eqref{eq:estut} and~\eqref{eq:estute}, we get
\begin{equation*}
J_1^2\le C \int_0^t  (t-s)^{-1/4} \|u_{\varepsilon}(s)- \bar{u}(s)\|_{\llf^2(\rr)}^2\,ds.
\end{equation*}

For $J_2$, as in \eqref{eq:estud}, we obtain
\begin{equation*}
J_2^2\le C \int_0^t  (t-s)^{-3/4} \|u_{\varepsilon}(s)- \bar{u}(s)\|_{\llf^2(\rr)}^2\,ds.
\end{equation*}

Therefore,
\begin{equation*}
\|u_{\varepsilon}(t)- \bar{u}(t)\|_{\llf^2(\rr)}^2
\le C \int_0^t  (t-s)^{-3/4} \|u_{\varepsilon}(s)- \bar{u}(s)\|_{\llf^2(\rr)}^2\,ds+3\|{\xi_\varepsilon}(t)\|_{\llf^2(\rr)}^2.
\end{equation*}

We use the generalized Gronwall's inequality (see, for example, Corollary~1 in~\cite{ye07gron})
and get
\begin{equation*}
\begin{split}
\|u_{\varepsilon}(t)- \bar{u}(t)\|_{\llf^2(\rr)}^2\le C(\omega)\|{\xi_\varepsilon}(t)\|_{\llf^2(\rr)}^2
+C(\omega)\int_0^t\sum_{n=1}^\infty \dfrac{\Gamma^n(1/4)}{\Gamma(n/4)}(t-s)^{n/4-1}\|{\xi_\varepsilon}(s)\|_{\llf^2(\rr)}^2\,ds.
\end{split}
\end{equation*}

By \eqref{eq:thnbound}, $\|{\xi_\varepsilon}(s)\|_{\llf^2(\rr)}\le C$. It is easy to check that the function
\[
h(s)=\sum_{n=1}^\infty \dfrac{\Gamma^n(1/4)}{\Gamma(n/4)}(t-s)^{n/4-1}
\]
is integrable. Applying \eqref{eq:esxigl} and the dominated convergence theorem, we obtain \eqref{eq:ueut}.

\end{proof}

\bibliographystyle{bib/vmsta-mathphys}
\bibliography{RadchenkoBurgersEq22}

\begin{thebibliography}{27}
\ifx \bisbn   \undefined \def \bisbn  #1{ISBN #1}\fi
\ifx \binits  \undefined \def \binits#1{#1}\fi
\ifx \bauthor  \undefined \def \bauthor#1{#1}\fi
\ifx \batitle  \undefined \def \batitle#1{#1}\fi
\ifx \bjtitle  \undefined \def \bjtitle#1{#1}\fi
\ifx \bvolume  \undefined \def \bvolume#1{\textbf{#1}}\fi
\ifx \byear  \undefined \def \byear#1{#1}\fi
\ifx \bissue  \undefined \def \bissue#1{#1}\fi
\ifx \bfpage  \undefined \def \bfpage#1{#1}\fi
\ifx \blpage  \undefined \def \blpage #1{#1}\fi
\ifx \burl  \undefined \def \burl#1{\textsf{#1}}\fi
\ifx \doiurl  \undefined \def \doiurl#1{\textsf{#1}}\fi
\ifx \betal  \undefined \def \betal{\textit{et al.}}\fi
\ifx \binstitute  \undefined \def \binstitute#1{#1}\fi
\ifx \binstitutionaled  \undefined \def \binstitutionaled#1{#1}\fi
\ifx \bctitle  \undefined \def \bctitle#1{#1}\fi
\ifx \beditor  \undefined \def \beditor#1{#1}\fi
\ifx \bpublisher  \undefined \def \bpublisher#1{#1}\fi
\ifx \bbtitle  \undefined \def \bbtitle#1{#1}\fi
\ifx \bedition  \undefined \def \bedition#1{#1}\fi
\ifx \bseriesno  \undefined \def \bseriesno#1{#1}\fi
\ifx \blocation  \undefined \def \blocation#1{#1}\fi
\ifx \bsertitle  \undefined \def \bsertitle#1{#1}\fi
\ifx \bsnm \undefined \def \bsnm#1{#1}\fi
\ifx \bsuffix \undefined \def \bsuffix#1{#1}\fi
\ifx \bparticle \undefined \def \bparticle#1{#1}\fi
\ifx \barticle \undefined \def \barticle#1{#1}\fi
\ifx \bconfdate \undefined \def \bconfdate #1{#1} \fi
\ifx \botherref \undefined \def \botherref #1{#1} \fi
\ifx \url \undefined \def \url#1{\textsf{#1}} \fi
\ifx \bchapter \undefined \def \bchapter#1{#1} \fi
\ifx \bbook \undefined \def \bbook#1{#1} \fi
\ifx \bcomment \undefined \def \bcomment#1{#1} \fi
\ifx \oauthor \undefined \def \oauthor#1{#1} \fi
\ifx \citeauthoryear \undefined \def \citeauthoryear#1{#1} \fi
\ifx \endbibitem  \undefined \def \endbibitem {}\fi
\ifx \bconflocation  \undefined \def \bconflocation#1{#1} \fi
\ifx \arxivurl  \undefined \def \arxivurl#1{\textsf{#1}} \fi
\csname PreBibitemsHook\endcsname

\bibitem{bodumj}
\begin{barticle}
\bauthor{\bsnm{Bodnarchuk}, \binits{I.}}:
\batitle{Regularity of the mild solution of a parabolic equation with
  stochastic measure}.
\bjtitle{Ukr. Math. J.}
\bvolume{69},
\bfpage{1}--\blpage{18}
(\byear{2017})
\end{barticle}
\endbibitem

\bibitem{bodrad20}
\begin{barticle}
\bauthor{\bsnm{Bodnarchuk}, \binits{I.}},
\bauthor{\bsnm{Radchenko}, \binits{V.}}:
\batitle{The wave equation in the three-dimensional space driven by a general
  stochastic measure}.
\bjtitle{Theor. Probability and Math. Statist.}
\bvolume{100},
\bfpage{43}--\blpage{60}
(\byear{2020})
\end{barticle}
\endbibitem

\bibitem{bodnarchuk_2020}
\begin{barticle}
\bauthor{\bsnm{Bodnarchuk}, \binits{I.}}:
\batitle{Averaging principle for a stochastic cable equation}.
\bjtitle{Mod. Stoch. Theory Appl.}
\bvolume{7}(\bissue{4}),
\bfpage{449}--\blpage{467}
(\byear{2020})
\end{barticle}
\endbibitem

\bibitem{dong07}
\begin{barticle}
\bauthor{\bsnm{Dong}, \binits{Z.}},
\bauthor{\bsnm{Xu}, \binits{T.G.}}:
\batitle{One-dimensional stochastic {B}urgers equation driven by {L}{\'e}vy
  processes}.
\bjtitle{J. Funct. Anal.}
\bvolume{243},
\bfpage{631}--\blpage{678}
(\byear{2007})
\end{barticle}
\endbibitem

\bibitem{dretop}
\begin{barticle}
\bauthor{\bsnm{Drewnowski}, \binits{L.}}:
\batitle{Topological rings of sets, continuous set functions,
  integration.~{III}}.
\bjtitle{Bull. Acad. Pol. Sci. S\'{e}r. sci. math. astron. phys.}
\bvolume{20},
\bfpage{439}--\blpage{445}
(\byear{1972})
\end{barticle}
\endbibitem

\bibitem{gyongy98}
\begin{barticle}
\bauthor{\bsnm{Gy{\"o}ngy}, \binits{I.}}:
\batitle{Existence and uniqueness results for semilinear stochastic partial
  differential equations}.
\bjtitle{Stochastic Process. Appl.}
\bvolume{73},
\bfpage{271}--\blpage{299}
(\byear{1998})
\end{barticle}
\endbibitem

\bibitem{gyonua}
\begin{barticle}
\bauthor{\bsnm{Gy{\"o}ngy}, \binits{I.}},
\bauthor{\bsnm{Nualart}, \binits{D.}}:
\batitle{On the stochastic {B}urgers equation in the real line}.
\bjtitle{Ann. Probab.}
\bvolume{27},
\bfpage{782}--\blpage{802}
(\byear{1999})
\end{barticle}
\endbibitem

\bibitem{gyonrov}
\begin{barticle}
\bauthor{\bsnm{Gy{\"o}ngy}, \binits{I.}},
\bauthor{\bsnm{Rovira}, \binits{C.}}:
\batitle{On stochastic partial differential equations with polynomial
  nonlinearities}.
\bjtitle{Stochastics}
\bvolume{67},
\bfpage{123}--\blpage{146}
(\byear{1999})
\end{barticle}
\endbibitem

\bibitem{jacob10}
\begin{barticle}
\bauthor{\bsnm{Jacob}, \binits{N.}},
\bauthor{\bsnm{Potrykus}, \binits{A.}},
\bauthor{\bsnm{Wu}, \binits{J.-L.}}:
\batitle{Solving a non-linear stochastic pseudo-differential equation of
  {B}urgers type}.
\bjtitle{Stochastic Proc. Appl.}
\bvolume{120},
\bfpage{2447}--\blpage{2467}
(\byear{2010})
\end{barticle}
\endbibitem

\bibitem{kwawoy}
\begin{bbook}
\bauthor{\bsnm{Kwapie\'{n}}, \binits{S.}},
\bauthor{\bsnm{Woyczy\'{n}ski}, \binits{W.A.}}:
\bbtitle{Random Series and Stochastic Integrals: Single and Multiple}.
\bpublisher{Birkh\"{a}user},
\blocation{Boston}
(\byear{1992})
\end{bbook}
\endbibitem

\bibitem{lewnua18}
\begin{barticle}
\bauthor{\bsnm{Lewis}, \binits{P.}},
\bauthor{\bsnm{Nualart}, \binits{D.}}:
\batitle{Stochastic {B}urgers' equation on the real line: regularity and moment
  estimates}.
\bjtitle{Stochastics}
\bvolume{90},
\bfpage{1053}--\blpage{1086}
(\byear{2018})
\end{barticle}
\endbibitem

\bibitem{tudor07}
\begin{barticle}
\bauthor{\bsnm{Maejima}, \binits{M.}},
\bauthor{\bsnm{Tudor}, \binits{C.}}:
\batitle{Wiener integrals with respect to the {H}ermite process and a
  non-central limit theorem}.
\bjtitle{Stochastic Anal. Appl.}
\bvolume{25},
\bfpage{1043}--\blpage{1056}
(\byear{2007})
\end{barticle}
\endbibitem

\bibitem{manikin_2022av}
\begin{barticle}
\bauthor{\bsnm{Manikin}, \binits{B.}}:
\batitle{Averaging principle for the one-dimensional parabolic equation driven
  by stochastic measure}.
\bjtitle{Mod. Stoch. Theory Appl.}
\bvolume{9}(\bissue{2}),
\bfpage{123}--\blpage{137}
(\byear{2022})
\end{barticle}
\endbibitem

\bibitem{mazz20}
\begin{barticle}
\bauthor{\bsnm{Mazzonetto}, \binits{S.}},
\bauthor{\bsnm{Salimova}, \binits{D.}}:
\batitle{Existence, uniqueness, and numerical approximations for stochastic
  {B}urgers equations}.
\bjtitle{Stochastic Anal. Appl.}
\bvolume{38},
\bfpage{623}--\blpage{646}
(\byear{2020})
\end{barticle}
\endbibitem

\bibitem{memiva}
\begin{barticle}
\bauthor{\bsnm{Memin}, \binits{T.}},
\bauthor{\bsnm{Mishura}, \binits{Y.}},
\bauthor{\bsnm{Valkeila}, \binits{E.}}:
\batitle{Inequalities for the moments of {W}iener integrals with respect to a
  fractional {B}rownian motion}.
\bjtitle{Statist. Probab. Lett.}
\bvolume{51},
\bfpage{197}--\blpage{206}
(\byear{2001})
\end{barticle}
\endbibitem

\bibitem{peszab}
\begin{bbook}
\bauthor{\bsnm{Peszat}, \binits{S.}},
\bauthor{\bsnm{Zabczyk}, \binits{J.}}:
\bbtitle{Stochastic Partial Differential Equations with {L}\'{e}vy Noise: An
  Evolution Equation Approach}.
\bpublisher{Cambridge University Press},
\blocation{Cambridge}
(\byear{2007})
\end{bbook}
\endbibitem

\bibitem{rads09}
\begin{barticle}
\bauthor{\bsnm{Radchenko}, \binits{V.}}:
\batitle{Mild solution of the heat equation with a general stochastic measure}.
\bjtitle{Studia Math.}
\bvolume{194},
\bfpage{231}--\blpage{251}
(\byear{2009})
\end{barticle}
\endbibitem

\bibitem{radavsm19}
\begin{barticle}
\bauthor{\bsnm{Radchenko}, \binits{V.}}:
\batitle{Averaging principle for equation driven by a stochastic measure}.
\bjtitle{Stochastics}
\bvolume{91},
\bfpage{905}--\blpage{915}
(\byear{2019})
\end{barticle}
\endbibitem

\bibitem{radavsh19}
\begin{barticle}
\bauthor{\bsnm{Radchenko}, \binits{V.}}:
\batitle{Averaging principle for the heat equation driven by a general
  stochastic measure}.
\bjtitle{Statist. Probab. Lett.}
\bvolume{146},
\bfpage{224}--\blpage{230}
(\byear{2019})
\end{barticle}
\endbibitem

\bibitem{radbook}
\begin{bbook}
\bauthor{\bsnm{Radchenko}, \binits{V.}}:
\bbtitle{General Stochastic Measures: {I}ntegration, Path Properties, and
  Equations}.
\bpublisher{Wiley -- ISTE},
\blocation{London}
(\byear{2022})
\end{bbook}
\endbibitem

\bibitem{samtaq}
\begin{bbook}
\bauthor{\bsnm{Samorodnitsky}, \binits{G.}},
\bauthor{\bsnm{Taqqu}, \binits{M.S.}}:
\bbtitle{Stable Non-{G}aussian Random Processes}.
\bpublisher{Chapman and Hall},
\blocation{London}
(\byear{1994})
\end{bbook}
\endbibitem

\bibitem{tudor08}
\begin{barticle}
\bauthor{\bsnm{Tudor}, \binits{C.}}:
\batitle{Analysis of the {R}osenblatt process}.
\bjtitle{ESAIM-Probab. Stat.}
\bvolume{12},
\bfpage{230}--\blpage{257}
(\byear{2008})
\end{barticle}
\endbibitem

\bibitem{tudor09}
\begin{barticle}
\bauthor{\bsnm{Tudor}, \binits{C.}}:
\batitle{On the {W}iener integral with respect to a sub-fractional {B}rownian
  motion on an interval}.
\bjtitle{J. Math. Anal. Appl.}
\bvolume{351},
\bfpage{456}--\blpage{468}
(\byear{2009})
\end{barticle}
\endbibitem

\bibitem{tudor13}
\begin{bbook}
\bauthor{\bsnm{Tudor}, \binits{C.}}:
\bbtitle{Analysis of Variations for Self-similar Processes: {A} Stochastic
  Calculus Approach}.
\bpublisher{Springer}
(\byear{2013})
\end{bbook}
\endbibitem

\bibitem{ye07gron}
\begin{barticle}
\bauthor{\bsnm{Ye}, \binits{H.}},
\bauthor{\bsnm{Gao}, \binits{J.}},
\bauthor{\bsnm{Ding}, \binits{Y.}}:
\batitle{A generalized {G}ronwall inequality and its application to a
  fractional differential equation}.
\bjtitle{J. Math. Anal. Appl.}
\bvolume{328}(\bissue{2}),
\bfpage{1075}--\blpage{1081}
(\byear{2007})
\end{barticle}
\endbibitem

\bibitem{yuan22}
\begin{barticle}
\bauthor{\bsnm{Yuan}, \binits{S.}},
\bauthor{\bsnm{Bl{\"o}mker}, \binits{D.}},
\bauthor{\bsnm{Duan}, \binits{J.}}:
\batitle{Stochastic turbulence for {B}urgers equation driven by cylindrical
  {L}{\'e}vy process}.
\bjtitle{Stoch. Dyn.}
\bvolume{22},
\bfpage{2240004}
(\byear{2022})
\end{barticle}
\endbibitem

\bibitem{zhou22}
\begin{barticle}
\bauthor{\bsnm{Zhou}, \binits{G.}},
\bauthor{\bsnm{Wang}, \binits{L.}},
\bauthor{\bsnm{Wu}, \binits{J.-L.}}:
\batitle{Global well-posedness of 2{D} stochastic {B}urgers equations with
  multiplicative noise}.
\bjtitle{Statist. Probab. Lett.}
\bvolume{182},
\bfpage{109315}
(\byear{2022})
\end{barticle}
\endbibitem

\end{thebibliography}

\end{document}